\documentclass[reqno]{amsart}
\usepackage{amsmath,amsthm,amssymb,amscd}

\newtheorem{theorem}{Theorem}[section]
\newtheorem{lemma}[theorem]{Lemma}

\newtheorem{corollary}[theorem]{Corollary}
\theoremstyle{definition}

\theoremstyle{remark}
\newtheorem{remark}[theorem]{Remark}

\numberwithin{equation}{section}
\allowdisplaybreaks
\begin{document}
	\title[Cliques of orders three and four in the Paley-type graphs]
	{Cliques of orders three and four in the Paley-type graphs}
	
	
	%    author one information
	%\author[short version for running head]{}
	\author{Anwita Bhowmik}
	\address{Department of Mathematics, Indian Institute of Technology Guwahati, North Guwahati, Guwahati-781039, Assam, INDIA}
	% \curraddr{}
	\email{anwita@iitg.ac.in}
	\author{Rupam Barman}
	\address{Department of Mathematics, Indian Institute of Technology Guwahati, North Guwahati, Guwahati-781039, Assam, INDIA}
	% \curraddr{}
	\email{rupam@iitg.ac.in}
	%\thanks{ }
	
	%    author two information

	%    \subjclass is required.
	\subjclass[2020]{05C25; 05C30; 11T24; 11T30}
	\date{January 18, 2023}
	\keywords{Paley graphs; Paley-type graphs; clique; finite fields; quadratic residue}
	%    Abstract is required.
	\begin{abstract}
		Let $n=2^s p_{1}^{\alpha_{1}}\cdots p_{k}^{\alpha_{k}}$, where $s=0$ or $1$, $\alpha_i\geq 1$, and the distinct primes $p_i$ satisfy $p_i\equiv 1\pmod{4}$ for all $i=1, \ldots, k$. Let $\mathbb{Z}_n^\ast$ denote the group of units in the commutative ring $\mathbb{Z}_n$. Recently, we defined a Paley-type graph $G_n$ of order $n$ as the graph whose vertex set is $\mathbb{Z}_n$ and $xy$ is an edge if $x-y\equiv a^2\pmod n$ for some $a\in\mathbb{Z}_n^\ast$. The Paley-type graph $G_n$ resembles the classical Paley graph in a number of ways, and adds to the list of generalizations of the Paley graph. Computing the number of cliques of a particular order in a Paley graph or its generalizations has been of considerable interest. For primes $p\equiv 1\pmod 4$ and $\alpha\geq 1$, by evaluating certain character sums, we found the number of cliques of order $3$ in $G_{p^\alpha}$ and expressed the number of cliques of order $4$ in $G_{p^\alpha}$ in terms of Jacobi sums. In this article we give combinatorial proofs and find the number of cliques of orders $3$ and $4$ in $G_n$ for all $n$ for which the graph is defined.
	\end{abstract}
	\maketitle
	\section{introduction and statements of results}
	The classical Paley graphs are a well-known infinite family of self complementary and symmetric graphs. Named after Raymond Paley, they were introduced as graphs independently by Sachs \cite{sachs} in 1962, and Erd\H{o}s and R\'enyi \cite{erdos} in 1963. Let $q$ be a prime power such that $q\equiv 1\pmod 4$ and let $\mathbb{F}_q$ denote the finite field with $q$ elements. The Paley graph of order $q$ is the graph with vertex set $\mathbb{F}_q$, where $ab$ is an edge if $a-b$ is a nonzero square in $\mathbb{F}_q$. Paley graphs are a special class of Cayley graphs. They are connected and strongly regular with parameters $(q, (q-1)/2, (q-5)/4, (q-1)/4)$. These graphs serve as a nice connection between graph theory and number theory and thus have generated a lot of interest in studying their properties, see for example \cite{ananchuen1993adjacency,chung}. 
	%Please see \cite{jones2017paley} for a detailed %survey on Paley graphs.
	\par Owing to their nice structure, Paley graphs have also been generalized in a number of ways. One such way is as follows. The vertex set of the graph in each case is $\mathbb{F}_q$ where $q$ is a prime power with some appropriate restrictions on $q$, depending on making the graph well-defined. The cubic Paley graph is such that $xy$ is an edge if $x-y$ is a nonzero cube in $\mathbb{F}_q$, while $xy$ is an edge in the quadruple Paley graph if $x-y$ is a nonzero fourth power in $\mathbb{F}_q$ (see \cite{ananchuen2006cubic} for more details). Further, Lim and Praeger \cite{lim2006generalised} defined the generalized $k$-Paley graph $G_k(q)$ for an integer $k\geq 2$, where $xy$ is an edge if $x-y$ is a nonzero $k$-th power in $\mathbb{F}_q$. Wage \cite{wage2006character} defined three graphs on $\mathbb{F}_p$, $p$ being a prime, by fixing some $t$ in $\mathbb{F}_p$: the undirected graphs $G_t(p)$ and $G_t'(p)$ are respectively such that $xy$ is an edge if $x-ty$ and $y-tx$ are both nonzero squares, and if one of them is a nonzero square. The directed graph $H_t(p)$ has an edge $x\rightarrow y$ if both $x$ and $x-ty$ are nonzero squares. Very recently, Reis \cite{reis} introduced another approach of generalization by studing a multiplicative-additive analogue arising from vector spaces over finite fields.
	\par  In \cite{BB}, we introduced a Paley-type graph which is defined as follows. Let $n=2^s p_{1}^{\alpha_{1}}\cdots p_{k}^{\alpha_{k}}$, where $s=0$ or $1$, $\alpha_i\geq 1$, and the distinct primes $p_i$ satisfy $p_i\equiv 1\pmod{4}$ for all $i=1, \ldots, k$. Let $\mathbb{Z}_n^\ast$ denote the group of units in the commutative ring $\mathbb{Z}_n$. Then, the Paley-type graph $G_n$ of order $n$ is the graph whose vertex set is $\mathbb{Z}_n$ and $ab$ is an edge if $a-b$ is a square in $\mathbb{Z}_n^\ast$. In \cite{angsu}, Das introduced Paley-type graphs $\Gamma_N$ modulo $N=pq$, where $p$ and $q$ are distinct primes satisfying $p\equiv q\equiv 1\pmod{4}$. The graph $\Gamma_N$ is a special case of the Paley-type graphs $G_n$, and it turns out to be $G_{pq}$. 
	\par The study of finding the number of cliques in Paley graphs and their generalizations has been of considerable interest. For a graph $G$, let $\mathcal{K}_n(G)$ denote the number of cliques of order $n$ in $G$. Evans et al. in \cite{evans1981number} gave a simple closed formula to calculate $\mathcal{K}_4(G(p))$, where $G(p)$ is the Paley graph of prime order $p$. In \cite{dawsey}, Dawsey and McCarthy computed $\mathcal{K}_4(G_k(q))$ in Lim and Praeger's graph $G_k(q)$ using finite field hypergeometric functions. In \cite{angsu}, Das gave a formula for $\mathcal{K}_3(G_{pq})$, where $p$ and $q$ are distinct primes satisfying $p\equiv q\equiv 1\pmod{4}$, and $G_{pq}$ is the Paley-type graph of order $pq$. For primes $p\equiv 1\pmod 4$ and $\alpha\geq 1$, in \cite{BB} we computed $\mathcal{K}_3(G_{p^\alpha})$ and $\mathcal{K}_4(G_{p^\alpha})$ for the Paley-type graph $G_{p^\alpha}$.
	\par The aim of this article is to find $\mathcal{K}_3(G_n)$ and $\mathcal{K}_4(G_n)$ for all $n$ for which the graph $G_n$ is defined. In \cite{BB}, by evaluating certain character sums, we found the number of cliques of order $3$ in $G_{p^\alpha}$ and expressed the number of cliques of order $4$ in $G_{p^\alpha}$ in terms of Jacobi sums. Here, we give combinatorial proofs to find $\mathcal{K}_3(G_n)$ and $\mathcal{K}_4(G_n)$ for any $n$. We omit the case when $n$ is even since there cannot exist cliques of order more than two in $G_n$ in that case. In the following theorem we find the number of cliques of order $3$ in $G_n$ for all odd $n$ for which the graph is defined. 
	\begin{theorem}\label{thm1}
	Let $n=p_{1}^{\alpha_{1}} \cdots p_{k}^{\alpha_{k}}$, where $\alpha_i\geq 1$ and the distinct primes $p_i$ satisfy $p_i\equiv 1\pmod{4}$ for all $i=1, \ldots, k$. Then, the number of cliques of order three in $G_n$ is given by 
	\begin{align*}
		\mathcal{K}_3(G_n)=\frac{1}{3\times 2^{3k+1}}\prod\limits_{i=1}^{k}\left[ p_i^{3\alpha_i-2}(p_i-1)(p_i-5)\right].
	\end{align*}	
	\end{theorem}
\begin{remark}
If we take $k=2$, $\alpha_1=\alpha_2=1$ and $p_1=5$ in Theorem \ref{thm1}, then we obtain Theorem 7 in \cite{angsu}.  Also, if we take $k=2$, $\alpha_1=\alpha_2=1$ in Theorem \ref{thm1}, then we obtain Theorem 12 in \cite{angsu}. 
\end{remark}
	In the following theorem, we provide a closed formula for the number of cliques of order four in $G_n$ for odd $n$. Note that if $p\equiv 1\pmod 4$ is a prime then there exist integers $a$ and $b$ such that $p=a^2+b^2$ where $a$ is even; $a^2$ and $b^2$ are unique.
	\begin{theorem}\label{thm2}
	Let $n=p_{1}^{\alpha_{1}} \cdots p_{k}^{\alpha_{k}}$, where $\alpha_i\geq 1$ and the distinct primes $p_i$ satisfy $p_i\equiv 1\pmod{4}$ for all $i=1, \ldots, k$. For $i\in\{1,\ldots, k\}$, let $p_i=a_i^2+b_i^2$ where $a_i,b_i$ are integers and $a_i$ is even. Then, the number of cliques of order four in $G_n$ is given by	
	\begin{align*}
		\mathcal{K}_4(G_n)=\dfrac{1}{3\times 8^{2k+1}}\prod\limits_{i=1}^k \left[p_i^{4\alpha_i-3}(p_i-1)\{(p_i-9)^2-4a_i^2\}\right].
	\end{align*}
	\end{theorem}
We find the values of $\mathcal{K}_3(G_n)$ and $\mathcal{K}_4(G_n)$ for some specific values of $n$ by using Python which are listed in Table \ref{Table-2}.
\begin{table}[ht]
	\begin{center}
\begin{tabular}{|c|c|c|}
	\hline
	$n$ & $\mathcal{K}_3(G_n)$&$\mathcal{K}_4(G_n)$\\
	\hline
	%$13$ & $26$&$0$\\
	$13^2$ & $57122$&$0$\\
	%$17$ & $68$&$0$\\
	$17^2$ & $334084$&$0$\\
	$13^2\times 17$ & $23305776$&$0$\\
	$29^2$ & $9901934$&$143578043$\\
	$29\times 37$ & $2163168$&$2703960$\\
	\hline
\end{tabular}
\caption{Values of $\mathcal{K}_3(G_n)$ and $\mathcal{K}_4(G_n)$}
\label{Table-2}
\end{center}
\end{table}
\par For the primes $p=13, 17, 29, 37$, we have $p=a^2+b^2$, where $a^2=4, 16, 4, 36$, respectively. By putting these values in Theorem \ref{thm1} and Theorem \ref{thm2}, we obtain the same values of $\mathcal{K}_3(G_n)$ and $\mathcal{K}_4(G_n)$ as listed in Table \ref{Table-2}.
\par In \cite{BB}, we computed $\mathcal{K}_4(G_{p^\alpha})$ using character sums. We state the result below for completeness.
\begin{theorem}[Theorem 1.2 in \cite{BB}]\label{thm3}
	Let $p$ be a prime such that $p\equiv 1\pmod{4}$. Let $\alpha\geq 1$ be an integer and let $n=p^\alpha$. Let $\chi$ denote the unique quadratic Dirichlet character mod $n$ and let $\psi$ be a Dirichlet character mod $n$ of order $4$. Let $J(\psi,\chi)=\sum\limits_{x\in\mathbb{Z}_n}\psi(x)\chi(1-x)$ be the Jacobi sum of $\psi$ and $\chi$. Then,
	$$\mathcal{K}_4(G_n)=\dfrac{p^{2\alpha-1}(p-1)[p^{2\alpha-2}\left\lbrace (p-9)^2-2p \right\rbrace+J(\psi,\chi)^2+\overline{J(\psi,\chi)}^2]}{1536}.$$
\end{theorem}
As a consequence of Theorem \ref{thm2}, we readily obtain the following formula for $\mathcal{K}_4(G_{p^\alpha})$ by taking $k=1$ which does not involve any Jacobi sums.
	\begin{corollary}\label{corr1}
		Let $p\equiv 1\pmod 4$ be a prime and let $\alpha$ be a positive integer. Let $p=a^2+b^2$ where $a$ and $b$ are integers such that $a$ is even. Then, the number of cliques of order four in $G_{p^\alpha}$ is given by
		\begin{align*}
			\mathcal{K}_4(G_{p^\alpha})=\dfrac{p^{4\alpha-3}(p-1)\{(p-9)^2-4a^2\}}{1536}.
		\end{align*}
	\end{corollary}
Let $p$ be a prime such that $p\equiv 1\pmod{4}$ and $\alpha$  a positive integer. Let $\chi$ denote the unique quadratic Dirichlet character mod $p^{\alpha}$ and let $\psi$ be a Dirichlet character mod $p^{\alpha}$ of order $4$. Let $J(\psi,\chi)=x+iy$ be the Jacobi sum. Clearly, $x$ and $y$ are integers. Combining Theorem \ref{thm3} and Corollary \ref{corr1}, we obtain
\begin{align}\label{xyreln}
	J(\psi,\chi)^2+\overline{J(\psi,\chi)}^2=2(x^2-y^2)=2p^{2\alpha-2}(p-2a^2).
\end{align}
Note that there are two Dirichlet characters modulo $p^\alpha$ of order $4$, namely $\psi$ and $\overline{\psi}$, and $J(\overline{\psi},\chi)=\overline{J(\psi,\chi)}=x-iy$. So, \eqref{xyreln} is independent of the choice of the characters of order $4$. In Table \ref{Table-1}, we calculate $x$ and $y$ by using Python and verify \eqref{xyreln} for some particular values of $p$ and $\alpha$.
\begin{table}[ht]
	\begin{center}
		\begin{tabular}{|c|c|c|c|c|c|}
			\hline
			$p,a^2,p-2a^2$ & $\alpha$ & $x$ & $y$ & $x^2-y^2$ & $p^{2\alpha-2}$ \\
			\hline
			$p=5=2^2+1^2,$ & $2$ & $5$ & $10$ & $5^2\times (-3)$ & $5^2$ \\ \cline { 2 - 6 } $a^2=4, p-2a^2=-3$& $3$ & $25$ & $50$ & $5^4\times (-3)$ & $5^4$\\
			\hline
			$p=13=2^2+3^2,$ & $2$ & $-39$ & $26$ & $13^2\times 5$ & $13^2$\\
			\cline { 2 - 6 }$a^2=4, p-2a^2=5$ & $3$ & $-507$ & $338$ & $13^4\times 5$ & $13^4$\\ 
			\hline 
			$p=17=4^2+1^2,$ & $2$ & $-17$ & $68$ & $17^2\times (-15)$ & $17^2$\\
			\cline { 2 - 6 }$a^2=16, p-2a^2=-15$ & $3$ & $-289$ & $1156$ & $17^4\times (-15)$ & $17^4$\\  
			\hline
			$p=29=2^2+5^2,$ & $1$ & $5$ & $2$ & $21$ & $1$ \\\cline { 2 - 6 } $a^2=4,p-2a^2=21$& $2$ & $145$ & $58$ & $29^2\times 21$ & $29^2$\\
			\hline 
			$p=37=6^2+1^2,$ & $1$ & $1$ & $-6$ & $-35$ & $1$ \\ \cline { 2 - 6 } $a^2=36,p-2a^2=-35$& $2$ & $37$ & $-222$ & $37^2\times (-35)$ & $37^2$\\
			\hline
			$p=41=4^2+5^2,$ & $1$ & $-5$ & $4$ & $9$ & $1$ \\ \cline { 2 - 6 } $a^2=16,p-2a^2=9$& $2$ & $-205$ & $164$ & $41^2\times 9$ & $41^2$\\
			\hline
		\end{tabular}
		\caption{Numerical data for \eqref{xyreln}}
		\label{Table-1}
	\end{center}
\end{table}
\begin{remark}
Let $p\equiv 1\pmod 4$ be a prime such that $p=a^2+b^2$ with $a$ even. Let $\chi$ and $\psi$ be characters modulo $p$ of orders $2$ and $4$ respectively. Then for suitably chosen signs of $a$ and $b$, the Jacobi sum $J(\psi,\chi)=b+ai$. For example, see \cite[Theorem 3.9]{berndt}. Thus, the identity \eqref{xyreln} is known if $\alpha=1$. We do not know whether the identity \eqref{xyreln} already exists in the literature when $\alpha>1$.
\end{remark}
We also prove the following corollary which follows from Theorem \ref{thm2}.
\begin{corollary}\label{corr2}
	Let $n=p_{1}^{\alpha_{1}} \cdots p_{k}^{\alpha_{k}}$, where $\alpha_i\geq 1$ and the distinct primes $p_i$ satisfy $p_i\equiv 1\pmod{4}$ for all $i=1, \ldots, k$.  Then, $\mathcal{K}_4(G_n)=0$ if and only if $p_i\in\{5,13,17\}$ for some $i\in\{1,\ldots,k\}$.
\end{corollary}
	The clique number of a graph is the order of a clique of maximum size contained in the graph. As a consequence of Theorem \ref{thm1} and Corollary \ref{corr2}, we readily obtain the following. 
	\begin{corollary}
			Let $n=p_{1}^{\alpha_{1}} \cdots p_{k}^{\alpha_{k}}$, where $\alpha_i\geq 1$ and the distinct primes $p_i$ satisfy $p_i\equiv 1\pmod{4}$ for all $i=1, \ldots, k$. We have:
	\begin{enumerate}
		\item if $p_i=5$ for some $i\in\{1,\ldots,k\}$ then the clique number of $G_n$ is $2$; and
		\item if $p_i\in\{13,17\}$ for some $i\in\{1,\ldots,k\}$ and $5\nmid n$ then the clique number of $G_n$ is $3$.
	\end{enumerate}	
\end{corollary}
\section{Preliminaries and some lemmas}
In this section, we prove two lemmas which will be required to prove the main results. For a positive integer $n$, let $\mathbb{Z}_n^\ast$ denote the group of units in the commutative ring $\mathbb{Z}_n$. Unless stated otherwise, $n$ denotes a positive integer such that $n=p_1^{\alpha_1}\cdots p_k^{\alpha_k}$, where $\alpha_i\geq 1$ and the distinct primes $p_i$ satisfy $p_i\equiv 1\pmod{4}$ for all $i=1, \ldots, k$. The following lemma states a relation between the squares in $\mathbb{Z}_n^\ast$ and the squares in $\mathbb{Z}_{p_i^{\alpha_i}}^\ast$, for $i\in\{1,\ldots,k\}$.
\begin{lemma}\label{lem1}
Let $n=p_1^{\alpha_1}\cdots p_k^{\alpha_k}$, where $\alpha_i\geq 1$ and the distinct primes $p_i$ satisfy $p_i\equiv 1\pmod{4}$ for all $i=1, \ldots, k$. Let $\mathbb{Z}_n^\ast$ denote the group of units in $\mathbb{Z}_n$, and let $R_n$ and $R_{p_i^{\alpha_i}}$ $(1\leq i\leq k)$ denote the group of squares in $\mathbb{Z}_n^\ast$ and $\mathbb{Z}_{p_i^{\alpha_i}}^\ast$ respectively. Then, 
for $x\in\mathbb{Z}_n$ we find that
\begin{align*}
	x\in R_n\text{ if and only if }x\in \bigcap\limits_{i=1}^k R_{p_i^{\alpha_i}}.
\end{align*} 
\end{lemma}
\begin{proof}
	Let $x\in\mathbb{Z}_n$ be such that $x\in R_n$. Then $x\equiv a^2\pmod n$ for some $a\in\mathbb{Z}_n^\ast$, which yields $x\equiv a^2\pmod {p_i^{\alpha_i}}$ for each $i\in\{1,\ldots,k\}$. Conversely, let $x\equiv a_1^2\pmod {p_1^{\alpha_1}},\ldots,x\equiv a_k^2\pmod {p_k^{\alpha_k}}$ for $a_i\in\mathbb{Z}_{p_i^{\alpha_i}}^\ast$. Let $z$ be an integer satisfying the system of congruences
	$$z\equiv a_1\pmod {p_1^{\alpha_1}},\ldots,z\equiv a_k\pmod {p_k^{\alpha_k}}.$$
	Then by Chinese Remainder Theorem we find that $x\equiv z^2\pmod n$ and so, $x\in R_n$. This completes the proof of the lemma.
\end{proof}
Before proving our next lemma, we state one from \cite{BB} which we use to prove our lemma.
\begin{lemma}[Lemma 2.6 in \cite{BB}]\label{lemBB}
	Let $p\equiv 1\pmod 4$ be a prime and let $\alpha\geq 1$ be an integer. We have ${\phi(n)/2\choose i} p^i\equiv 0\pmod{p^\alpha}$ where $n=p^\alpha,1\leq i\leq {\alpha-1}$. 
\end{lemma}
For primes $p\equiv 1\pmod 4$ and positive integers $\alpha$, the following lemma lists the squares in $\mathbb{Z}_{p^\alpha}^\ast$ in terms of the squares in $\mathbb{Z}_p^\ast$.
\begin{lemma}\label{lem2}
Let $p$	be a prime satisfying $p\equiv 1\pmod 4$, and let $\alpha$ be a positive integer. Let $R_p$ denote the set of nonzero squares in $\mathbb{Z}_p$, say $R_p:=(\mathbb{Z}_p^\ast)^2=\{r_1,\ldots,r_{\frac{p-1}{2}}\}$. Then the set of units in $\mathbb{Z}_{p^\alpha}$ which are squares is given by $$R_{p^\alpha}=\bigcup\limits_{i=1}^{\frac{p-1}{2}}\{r_i+tp:t\text{ is an integer satisfying }0\leq t\leq p^{\alpha-1}-1\}.$$
\end{lemma}
\begin{proof}
Consider $r_i\in R_p$ for some $i\in\{1,\ldots,\frac{p-1}{2}\}$, and $0\leq t\leq p^{\alpha-1}-1$. We show that $(r_i+tp)^{\frac{\phi(p^\alpha)}{2}}\equiv 1\pmod {p^\alpha}$, where $\phi$ denotes the Euler totient function. By the binomial theorem, we find that
\begin{align*}
 (r_i+tp)^{\frac{\phi(p^\alpha)}{2}}=\sum\limits_{m=0}^{\frac{\phi(p^\alpha)}{2}}\binom{\frac{\phi(p^\alpha)}{2}}{m}r_i^{\frac{\phi(p^\alpha)}{2}-m}(tp)^m.	
\end{align*}
Then, using Lemma \ref{lemBB} we have $(r_i+tp)^{\frac{\phi(p^\alpha)}{2}}\equiv r_i^{\frac{\phi(p^\alpha)}{2}}\pmod{p^\alpha}$. So, we proceed to show that $r_i^{\frac{\phi(p^\alpha)}{2}}\equiv 1\pmod {p^\alpha}$. Now, $r_i\in R_p$ implies $r_i^{\frac{p-1}{2}}\equiv 1\pmod p$, so we have $r_i^{\frac{p-1}{2}}=1+pX$ for some integer $X$. Again applying the binomial theorem, we find that
\begin{align*}
	(r_i^{\frac{p-1}{2}})^{p^{\alpha-1}}=(1+pX)^{p^{\alpha-1}}=\sum\limits_{l=0}^{p^{\alpha-1}}\binom{p^{\alpha-1}}{l}p^{l} X^{l}.
\end{align*}
So, it suffices to show that if $1\leq l\leq \alpha-1$ then $p^{\alpha-l}\mid \binom{p^{\alpha-1}}{l}$. We proceed along similar lines as in the proof of Lemma \ref{lemBB}. Let $v_p(y)$ denote the highest power of $p$ dividing $y$ and let $\sigma_p(y)$ denote the sum of digits of the base-$p$ representation of $y$. We have
\begin{align*}
	{\binom{p^{\alpha-1}}{l}}&=\dfrac{p^{\alpha-1}(p^{\alpha-1}-1)\cdots (p^{\alpha-1}-l+1)}{l!}\\
	&=\dfrac{p^{\alpha-l}p^{l-1}(p^{\alpha-1}-1)\cdots (p^{\alpha-1}-l+1)}{l!}.
\end{align*}
If possible, let $p^l\mid l!$. Then $v_p(l!)\geq l$, which implies $\frac{l-\sigma_p(l)}{p-1}\geq l$, which is not possible. So, $p^{\alpha-l}\mid \binom{p^{\alpha-1}}{l}$ and hence $r_i^{\frac{\phi(p^\alpha)}{2}}\equiv 1\pmod {p^\alpha}$. Thus, we have proved that $(r_i+tp)^{\frac{\phi(p^\alpha)}{2}}\equiv 1\pmod {p^\alpha}$.  Since $\mathbb{Z}_{p^\alpha}^\ast$ is cyclic of order $\phi(p^\alpha)$, this implies $r_i+tp\in R_{p^\alpha}$. Conversely, let $z\in R_{p^\alpha}$. Then $z\equiv x^2\pmod {p^\alpha}$ for some unit $x\in\mathbb{Z}_{p^\alpha}$, which implies $z\equiv x^2\pmod p$. Thus, $z\equiv r_i\pmod p$ for some $i\in\{1,\ldots,\frac{p-1}{2}\}$ and our proof is complete.
\end{proof}
\begin{remark}\label{rk}
	Let $n=p_{1}^{\alpha_{1}} \cdots p_{k}^{\alpha_{k}}$, where $\alpha_i\geq 1$ and the distinct primes $p_i$ satisfy $p_i\equiv 1\pmod{4}$ for all $i=1, \ldots, k$. Let $x\in\mathbb{Z}_n$. Then by Lemmas \ref{lem1} and \ref{lem2}, we have $$x\in R_n\text{ if and only if } x\pmod {p_i}\in R_{p_i} \text{ for }i=1,\ldots,k.$$
	This yields
	$$|R_n|=\prod\limits_{i=1}^{k}\left( p_i^{\alpha_i-1}\times\frac{p_i-1}{2}\right)=\frac{\phi(n)}{2^k} .$$
\end{remark}
\section{Proofs of Theorem \ref{thm1} and Theorem \ref{thm2}}
In this section, we prove our main results. We recall that $n=p_1^{\alpha_1}\cdots p_k^{\alpha_k}$, where $\alpha_i\geq 1$ and the distinct primes $p_i$ satisfy $p_i\equiv 1\pmod{4}$ for all $i=1, \ldots, k$. Let $R_n$ denote the group of squares in $\mathbb{Z}_n^\ast$ and let $H_n$ denote the subgraph of $G_n$ induced by $R_n$. As observed in Proposition 3.6 in \cite{BB}, $G_n$ is vertex-transitive. Let $r\in R_n$ be fixed. Then, the map $x\mapsto rx$ on $R_n$ gives a graph isomorphism on $H_n$. This proves that $H_n$ is also vertex-transitive. 
\par For a graph $G$ and a vertex $a\in G$, let $\mathcal{K}_l(G,a)$ denote the number of cliques of order $l$ in $G$ containing $a$. Now, let $l\geq 3$. Since $G_n$ is vertex-transitive, we find that
\begin{align}\label{eq1}
\mathcal{K}_l(G_n)&=\frac{n}{l}\times \mathcal{K}_l(G_n,0)\notag\\	
&=\frac{n}{l}\times \mathcal{K}_{l-1}(H_n).
\end{align}
Since $H_n$ is also vertex-transitive, we have
\begin{align}\label{eq2}
	\mathcal{K}_{l-1}(H_n)=\frac{|R_n|}{l-1}\times \mathcal{K}_{l-1}(H_n,1).
\end{align}
Combining \eqref{eq1} and \eqref{eq2}, we obtain
\begin{align}\label{eq3}
	\mathcal{K}_l(G_n)=\frac{n\phi(n)}{2^k l(l-1)}\times \mathcal{K}_{l-1}(H_n,1).
\end{align}
We are now ready to prove our first main result.
\begin{proof}[Proof of Theorem \ref{thm1}]
	Employing \eqref{eq3}, we find that the number of triangles in $G_n$ is given by
	\begin{align}\label{eq4}
	\mathcal{K}_3(G_n)=\frac{n\phi(n)}{3\times 2^{k+1}}\times \mathcal{K}_2(H_n,1).
	\end{align}
Let $E$ denote the set of edges in $H_n$ containing the vertex $1$, and for $i\in\{1,\ldots,k\}$ let $E_i$ denote the set of edges in $H_{p_i^{\alpha_i}}$ containing the vertex $1$. Let $(1,x)$ denote the edge connecting the vertices $1$ and $x$. We define a function $f:E\rightarrow E_1\times\cdots\times E_k$ as follows.
\begin{align*}
	(1,x)&\mapsto ((1\pmod{p_1^{\alpha_1}},x\pmod{p_1^{\alpha_1}}),\ldots,(1\pmod{p_k^{\alpha_k}},x\pmod{p_k^{\alpha_k}})).
\end{align*}
 $f$ is well defined, and using Chinese Remainder Theorem and Lemma \ref{lem1} we find that $f$ is one-one and onto as well. Thus, $f$ is a bijection which implies $|E|=\prod\limits_{i=1}^k |E_i|$. So, \eqref{eq4} yields
\begin{align}\label{eq5}
	\mathcal{K}_3(G_n)=\frac{n\phi(n)}{3\times 2^{k+1}}\times \prod\limits_{i=1}^k \mathcal{K}_2(H_{p_i^{\alpha_i}},1).
\end{align}
Now, in (4.15) of \cite{BB}, $\mathcal{K}_2(H_{p_i^{\alpha_i}},1)$ is expressed in terms of character sums. It was computed therein and it equals $\frac{p_i^{\alpha_i-1}(p_i-5)}{4}$. Hence, substituting this expression in \eqref{eq5} yields the required result.
\end{proof}
We now prove a lemma which will be used to prove Theorem \ref{thm2}.
The lemma counts the number of triangles of a particular kind in the graph $G_{p^\alpha}$ in terms of the number of triangles in $G_p$.
\begin{lemma}\label{lem4}
	Let $p\equiv 1\pmod 4$ be a prime and let $\alpha$ be a positive integer. Let $G_{p^\alpha}$ denote the Paley-type graph of order $p^\alpha$ and let $H_{p^\alpha}$ denote its subgraph induced by the set of units which are squares in $\mathbb{Z}_{p^\alpha}$. Similarly, let $G_{p}$ denote the Paley graph of order $p$ and let $H_{p}$ denote its subgraph induced by the nonzero squares in $\mathbb{Z}_{p}$. Let $p=a^2+b^2$, where $a$ and $b$ are integers such that $a$ is even. Then, the number of cliques of order three in $H_{p^\alpha}$ containing the vertex $1$ is given by
	$$\mathcal{K}_3(H_{p^\alpha},1)=p^{2\alpha-2}\times \mathcal{K}_3(H_{p},1)=p^{2\alpha-2}\times\left( \dfrac{(p-9)^2-4a^2}{64}\right) .$$
\end{lemma}
\begin{proof}
	Let $(1,x,y)$ denote a triangle in $H_{p^\alpha}$ containing the vertex $1$, where $x,y\in\mathbb{Z}_{p^\alpha}$. Then, $x,y,x-1,y-1,x-y\in R_{p^\alpha}$ and so, $(1\pmod p, x\pmod p, y\pmod p)$ yields a triangle in $H_p$ containing the vertex $1$. Conversely, let $(1,c,d)$ denote a triangle in $H_p$ containing the vertex $1$, whereby $c,d,c-1,d-1,c-d\in R_p$. Employing Lemma \ref{lem2}, we find that $(1,c+tp,d+t'p)$ yields a triangle in $H_{p^\alpha}$ containing the vertex $1$, where $t,t'$ are integers satisfying $0\leq t,t'\leq p^{\alpha-1}-1$. Each such triangle constructed in $H_{p^\alpha}$ is unique: if $(1,c+tp,d+t'p)$ and $(1,d+t''p,c+t'''p)$ yield the same triangle ($0\leq t,t',t'',t'''\leq p^{\alpha-1}-1$), this implies $c\equiv d\pmod p$ which is a contradiction. Thus, the triangle $(1,c,d)$ in $H_p$ produces $(p^{\alpha-1})^2$ distinct triangles in $H_{p^\alpha}$ containing the vertex $1$. We have proved the first equality in the statement of the lemma. The second equality follows from Proposition 4 in \cite{evans1981number}. This completes the proof of the lemma.
\end{proof}
\begin{proof}[Proof of Theorem \ref{thm2}]
	We proceed as in the proof of Theorem \ref{thm1} and employ \eqref{eq3}. We have
	\begin{align}\label{eq6}
		\mathcal{K}_4(G_n)=\frac{n\phi(n)}{3\times 2^{k+2}}\times \mathcal{K}_3(H_n,1).
	\end{align}
Let $F$ denote the set of triangles in $H_n$ containing the vertex $1$, and for $i\in\{1,\ldots,k\}$ let $F_i$ denote the set of triangles in $H_{p_i^{\alpha_i}}$ containing the vertex $1$. Let us set the notation $(1,x,y)$ to denote the triangle with the vertices $1,x$ and $y$. We define the following map:
\begin{align*}
	g:F&\rightarrow F_1\times\cdots\times F_k\\
	(1,x,y)&\mapsto ((1\pmod{p_1^{\alpha_1}},x\pmod{p_1^{\alpha_1}},y\pmod{p_1^{\alpha_1}}),\ldots,\\
	&\hspace{.65cm}(1\pmod{p_k^{\alpha_k}},x\pmod{p_k^{\alpha_k}},y\pmod{p_k^{\alpha_k}})).
\end{align*}
$g$ is a well defined function. Next, to show that $g$ is onto we seek for a pre-image of an element in $F_1\times\cdots\times F_k$, say $j:=((1,c_1,d_1),\ldots,(1,c_k,d_k))$. We find $x,y\in\mathbb{Z}_n$ such that $(1,x,y)\in F$. We have a solution for $x$ and $y$ if they satisfy the following system of congruences:
\begin{align}\label{eq7}
	&x\equiv c_1\pmod{p_1^{\alpha_1}},\ldots,x\equiv c_k\pmod{p_k^{\alpha_k}},\notag\\
	&y\equiv d_1\pmod{p_1^{\alpha_1}},\ldots,y\equiv d_k\pmod{p_k^{\alpha_k}}.
\end{align}
Let us denote this system as $\left( \begin{array}{ccc}
	c_1 & \ldots & c_k\\ d_1 & \ldots & d_k
\end{array}\right) $. Note that interchanging the elements in any column of $\left( \begin{array}{ccc}
c_1 & \ldots & c_k\\ d_1 & \ldots & d_k
\end{array}\right) $ and finding the solution to the corresponding system of congruences also gives another pre-image of $j$ in $F$. However, $(1,c_i,d_i)$ and $(1,d_i,c_i)$ denote the same triangle, so the system $\left( \begin{array}{ccc}
d_1 & \ldots & d_k\\ c_1 & \ldots & c_k
\end{array}\right) $ gives the same triangle in $H_{p^\alpha}$ as \eqref{eq7} does. Hence, each element in $F_1\times\cdots\times F_k$ has $\frac{2^k}{2}=2^{k-1}$ pre-images in $F$, and therefore \eqref{eq6} yields
\begin{align}\label{eq8}
	\mathcal{K}_4(G_n)=\frac{n\phi(n)}{3\times 2^{k+2}}\times 2^{k-1}\times \prod\limits_{i=1}^k \mathcal{K}_3(H_{p_i^{\alpha_i}},1).
\end{align}
 Finally, employing Lemma \ref{lem4} to substitute $\mathcal{K}_3(H_{p_i^{\alpha_i}},1)$ for $i\in\{1,\ldots,k\}$ in \eqref{eq8} completes the proof of the theorem.
\end{proof}
We now give a proof of Corollary \ref{corr2} as a consequence of Theorem \ref{thm2}.
\begin{proof}[Proof of Corollary \ref{corr2}]
	From Theorem \ref{thm2} we find that $\mathcal{K}_4(G_n)=0$ if and only if $(p_i-9)^2-4 a_i^2=0$ for some $i\in\{1,\ldots,k\}$, where $p_i=a_i^2+b_i^2$ for integers $a_i$ and $b_i$ such that $a_i$ is even. So, we have $(p_i-9)^2-4 a_i^2=0$ if and only if $(p_i-9)^2-4(p_i-b_i^2)=0$. Solving this equation as a quadratic equation in $p_i$, we have $p_i^2-22p_i+(81+4 b_i^2)=0$, and readily obtain $p_i=11\pm 2\sqrt{10-b_i^2}$. Then we have the inequality $p_i\leq 11+2\sqrt{10}\approx 17.3245\ldots$ and so, $p_i\leq 17$. The only primes $p\equiv 1\pmod 4$ satisfying $p\leq 17$ are $5,13$ and $17$. Finally, we finish the proof by observing that for $p=13$ and $p=17$, if $p=a^2+b^2$ for integers $a$ and $b$ such that $a$ is even, then $(p-9)^2-4a^2=0$.
\end{proof}

\end{document}